\documentclass[12pt]{article}
\usepackage{amsmath,amsxtra,latexsym,amsthm,amssymb,amscd,cite,amsfonts}
\usepackage[portrait, top=3.5cm, bottom=3cm, left=3.5cm, right=2cm] {geometry}
\usepackage{xcolor}
\usepackage{graphicx}
\theoremstyle{plain}
\newtheorem{ex}{\bf Example}[section]
\newtheorem{Remark}{\bf Remark}[section]
\newtheorem{Proposition}{\bf Proposition}[section]
\newtheorem{Definition}{\bf Definition}[section]
\newtheorem{Theorem}{\bf Theorem}[section]

\def\tto{\;{\lower 1pt \hbox{$\rightarrow$}}\kern -10pt
\hbox{\raise 2pt \hbox{$\rightarrow$}}\;}

\def\Bar{\overline}

\def\gph{\mbox{\rm gph}\,}

\def\dom{\mbox{\rm dom}\,}

\begin{document}
\pagestyle{myheadings}

\renewcommand{\theequation}{\thesection.\arabic{equation}}
\normalsize

\setcounter{equation}{0}

\title{Differential stability of convex optimization problems with possibly empty solution sets}

\author{D.T.V.~An\footnote{Department of Mathematics and Informatics, Thai Nguyen University of Sciences, Thai Nguyen city, Vietnam; email:
andtv@tnus.edu.vn.}\ \, and\ \ J.-C. Yao\footnote{ Center for General Education, China Medical University, 
			Taichung 40402, Taiwan; email: yaojc@mail.cmu.edu.tw.}}
\maketitle
\date{}

\medskip
\begin{quote}
\noindent {\bf Abstract.} 
As a complement to two recent papers by An and Yen \cite{AnYen2015}, and by An and Yao \cite{AnYao2016} on subdifferentials of the optimal value function of infinite-dimensional convex optimization problems, this paper studies the differential stability of convex optimization problems, where the solution set may be empty. By using a suitable sum rule for $\varepsilon$-subdifferentials, we obtain exact formulas for computing the $\varepsilon$-subdifferential of the optimal value function. Several illustrative examples are also given.



\end{quote}

\section{Introduction}
\markboth{\centerline{\it Introduction}}{\centerline{\it D.T.V.~An
and J.-C.~Yao}} \setcounter{equation}{0}
Studying \textit{differential stability} of optimization problems usually means to study differentiability properties of the optimal value
function in parametric mathematical programming. We refer to  \cite{AnYao2016,AnYen2015,Mordukhovich_2006,Mordukhovich_Nam_Rector_Tran,MordukhovichEtAl_2009,Penot2013,Rockafellar_1970,Zanlinescu_2002} and the references therein for some old and new results in this direction.  

\medskip
According to Penot \cite[Chapter~3]{Penot2013}, the class of convex functions is an important class that enjoys striking and useful properties. The consideration of  \textit{directional derivative} makes it possible to reduce this class to the subclass of sublinear functions. This subclass is next to the family of linear functions in terms of simplicity: the epigraph of a sublinear function is a convex cone, a notion almost as simple and useful as the notion of linear subspace.

\medskip
Differential properties of convex functions
have been studied intensively in the last five decades. The fundamental contributions of J.-J.~Moreau and R.T.~Rockafellar have been widely recognized. Their results led to the beautiful theory of convex analysis \cite{Rockafellar_1970}. The derivative-like structure for convex functions, called \textit{subdifferentials}, is one of the main concepts in this theory. Subdifferentials generalize the derivatives to nonsmooth functions, which make them one of the most useful instruments in nonsmooth optimization.
\medskip

The concept of the \textit{$\varepsilon$-subdifferential} or\textit{ approximate subdifferential}  was first introduced by Br\o ndsted and Rockafellar in  \cite{Rockafellar_Brondsted_1965}. It has become an essential tool in convex analysis. For example, approximate minima and approximate subdifferentials are linked together by Legendre -Fenchel transforms (see, e.g.,\cite{Volle1994}).
Like for the subdifferential, calculus rules on the $\varepsilon$-subdifferential are of importance and attract the attention of many researchers; see, e.g., \cite{Hiriart_Urruty_1982,Hiriart_Urruty_1989,Hiriart_Moussaoui_Seeger_1995,Hiriart_Lemarechal_1993,Hiriart_Phelps_1993,MoussaouiSeeger1994,Seeger1994,Seeger1996,Volle1994,Zanlinescu_2002} and the references therein.

\medskip
In \cite{AnYen2015}, An and Yen presented formulas for computing the subdifferential of the optimal value function of convex optimization problems under inclusion constraints in a Hausdorff locally convex topological vector space setting. Afterwards, An and Yao~\cite{AnYao2016} obtained new results on subdifferential of the just mentioned function for problems under geometrical and functional constraints in Banach spaces. In both papers, the authors assumed that the original convex program has a nonempty solution set. A natural question arises: \textit{Is there any analogous version of the formulas given in \cite{AnYen2015,AnYao2016} for the case where the solution set can be empty?} 

\medskip
By using sum rules of the $\varepsilon$-subdifferentials from \cite{Hiriart_Urruty_1982} and appropriate regularity conditions, this paper presents formulas for the  $\varepsilon$-subdifferential of the optimal value function of convex optimization problems under inclusion constraints in Hausdorff locally convex topological vector spaces.

\medskip
The contents of the paper are as follows. Section~2 recalls several definitions and elementary results related to $\varepsilon$-subdifferentials of convex functions. Section~3 is devoted to a detailed analysis of several sum rules for $\varepsilon$-subdifferentials. Differential stability results of unconstrained and constrained convex optimization problems are established in Section~4. Several illustrative examples are also presented in this section.

\section{Preliminaries}
\markboth{\centerline{\it Preliminaries}}{\centerline{\it D.T.V.~An
and J.-C. Yao
}} \setcounter{equation}{0}
	Let $X$ and $Y$ be Hausdorff locally convex topological vector spaces whose topological duals are denoted, respectively, by $X^*$ and $Y^*$. Let $f:X \rightarrow \Bar{\Bbb{R}}$, where $\Bar{\Bbb{R}}:= [-\infty, + \infty ]= \mathbb{R} \cup \{+ \infty\}\cup \{-\infty\}$ is  an extended real-valued function.  One says that $f$ is \textit{proper} if the \textit{domain}
 		   $${\rm{dom}}\, f:=\{ x \in X \mid f(x) < +\infty\} $$
 		   is nonempty, and if $f(x) > - \infty$ for all $x \in X$. It is well known that if ${\rm{epi}}\,f$ of $f$ is convex, then $f$ is said to be a convex function, where $$ {\rm{epi}}\, f:=\{ (x, \alpha) \in X \times \mathbb{R} \mid \alpha \ge f(x)\}.$$ If ${\rm{epi}}\, f$ is a closed subset of $X \times \Bbb{R}$, $f$ is said to be a {\it closed} function. Denoting the set of all the neighborhoods of $x$ by $\mathcal{N}(x)$, one says that $f$ is {\it lower semicontinuous} (l.s.c.) at $x\in X$ if for every $\varepsilon >0$ there exists $U\in \mathcal{N}(x)$ such that $f(x') \ge f(x)-\varepsilon$ for any $x' \in U.$ If $f$ is l.s.c. at every $x\in X$, $f$ is said to be l.s.c. on $X$. It is easy to show that: {\it $f$ is l.s.c. on $X$ if and only if $f$ is closed and ${\rm dom}\, f$ is closed too}. 

\medskip
It is convenient to denote the set of all proper lower semicontinuous convex functions on $X$ by $\Gamma_0(X)$.
\begin{Definition}\rm
Let $f$ be a convex function defined on $X$, $\bar x\in {\rm dom}\,f$, and $\varepsilon \ge 0$. The $\varepsilon$-\textit{subdifferential} of $f$ at $\bar x$ is the set
$$ \partial_\varepsilon f(\bar x)=\{x^*\in X^* \mid \langle x^*, x-\bar x \rangle \le f(x)-f(\bar x)+\varepsilon, \ \, \forall x\in X \}.$$
\end{Definition}
The set $\partial_\varepsilon f(\bar x)$ reduces to the subdifferential $\partial f(\bar x)$ when $\varepsilon=0$. From the definition it follows that $\partial_\varepsilon f(\bar x)$ is a weakly$^*$-closed, convex set. In addition, for any nonnegative values $\varepsilon_1,\, \varepsilon_2$ with $\varepsilon_1 \le \varepsilon_2,$ one has $\partial_{\varepsilon_1} f(\bar x) \subset \partial_{\varepsilon_2} f(\bar x)$. Moreover,
$$\partial f(\bar x)=\partial_0 f(\bar x)=\bigcap\limits_{\varepsilon>0}\partial_\varepsilon f(\bar x).$$

If $f \in \Gamma_0(X)$, then $\partial_\varepsilon f(\bar x)$ is nonempty for every $\bar x \in {\rm dom}\,f$ and $\varepsilon > 0$ (see~\cite{Hiriart_Urruty_1982}). The following example shows that the traditional subdifferential $\partial f(\bar x)$ may be empty, while $\partial_\varepsilon f(\bar x)\neq\emptyset$ for all $\varepsilon > 0$. 
\begin{ex}\label{Ex_sqrt_minus}\rm
Let $X= \mathbb{R}$ and $\bar x=0$. Clearly, the function $f:X\to \overline{\mathbb{R}}$ given by
$$f(x)=\begin{cases}
- \sqrt{x}&\mbox{if}\ x \ge 0,\\
+\infty & \mbox{otherwise}
\end{cases}$$ belongs to $\Gamma_0(X)$ and $\bar x \in {\rm dom}\,f$.
For every $\varepsilon > 0$, one has
\begin{align*}
\partial_\varepsilon f(\bar x)&= \{x^* \in X^* \mid \langle x^*, x -\bar x \rangle \le f(x)-f(\bar x) + \varepsilon , \ \, \forall x \in X\}\\
&= \left\{x^* \in \mathbb{R} \mid  x^* x \le - \sqrt{x} + \varepsilon , \ \, \forall x \ge 0 \right\}\\
&=\left( -\infty,\, -\frac{1}{4\varepsilon}\right].
\end{align*}
Meanwhile, it is easy to verify that $\partial f(\bar x)= \emptyset.$
\end{ex}

In the sequel, we will also need the notion of conjugate function. By definition, the function $f^*: X^* \rightarrow  \overline{\Bbb{R}}$ given by
\begin{align*}
f^*(x^*)= \sup\limits_{x \in X} \left[ \langle x^*,x \rangle - f(x) \right], \quad x^* \in X^*,
\end{align*}
is said to be the \textit{conjugate function} (also called the \textit{Young--Fenchel transform}, the \textit{Legendre--Fenchel conjugate}) of $f:X\to \overline{\Bbb{R}}$. The conjugate function of $f^*$, denoted by $f^{**}$, is a function defined on $X$ and has values in $\overline{\Bbb{R}}$:
$$f^{**}(x)=\sup\limits_{x^* \in X^*}\left[ \langle x^*,x \rangle -f^*(x^*)\right] \quad (x \in X) .$$
Clearly, the function $f^{**}$ is convex and closed (in the sense that ${\rm epi}\,f^{**}$ is closed in the weak topology of $X\times \mathbb{R}$ or, in other words, $f^{**}$ is lower semicontinuous w.r.t. the weak topology of $X$). According to the Fenchel--Moreau theorem (see \cite[Theorem~1, p.~175]{Ioffe_Tihomirov_1979}), \textit{if $f$ is a function on $X$ everywhere greater than $-\infty$, then $f=f^{** }$ if and only if $f$ is closed and convex.}

\medskip
According to \cite{Hiriart_Urruty_1982}, there are two basic ways to describe $\partial_\varepsilon f(\bar x)$: 
\par (a) Via the conjugate function $f^*$ of $f$;

\par (b) Via the support function $\delta^*(x;\partial_\varepsilon f(\bar x)):=\sup\{\langle x^*,x\rangle\mid x^* \in \partial_\varepsilon f(\bar x) \}$ of $\partial_\varepsilon f(\bar x)$.

\begin{Proposition}\label{Proposition_1} {\rm({See \cite[Propositions 1.1 and 1.2]{Hiriart_Urruty_1982}})} The following holds:
	
		{\rm (i)} If $\bar x\in {\rm dom}\,f$ and $\varepsilon \ge 0$, then
		$$x^*\in \partial_\varepsilon f(\bar x)\ \, \Longleftrightarrow\ \, f^*(x^*)+f(\bar x)\le \langle x^*, \bar x \rangle +\varepsilon. $$
		
	{\rm (ii)} If $f\in \Gamma_0(X),$ $\bar x\in {\rm dom}\,f$ and $\varepsilon \ge 0$, then
		$$\delta^*(v; \partial_\varepsilon f(\bar x))=\inf_{t>0}\dfrac{f(\bar x+tv)-f(\bar x)+\varepsilon}{t}\quad\, (v\in X).$$
\end{Proposition}

To deal with constrained optimization problems, we will need some results on $\varepsilon$-normal directions from \cite{Hiriart_Urruty_1989}. Let $C$ be a nonempty convex set in a Hausdorff locally convex topological vector space $X$. 
\begin{Definition} \rm The set $N_\varepsilon(\bar x;C)$ of $\varepsilon$-\textit{normal directions} to $C$ at  $\bar x \in C$ is defined by
	\begin{align*}
	N_\varepsilon(\bar x;C)=\{x^*\in X^*\mid \langle x^*, x-\bar x \rangle \le \varepsilon, \ \forall x\in C \}. 
	\end{align*}
\end{Definition}

As usual, the indicator function $\delta(\cdot; C)$ of $C$ is defined by setting $\delta (x; C)=0$ if $x \in C$ and $\delta (x; C)=+\infty$ if $x \notin C$. It is easy to see that $N_\varepsilon(\bar x; C)= \partial_{\varepsilon}\delta (\bar x; C) $ for every $\varepsilon \ge 0$. Moreover, when $\varepsilon=0$, $N_\varepsilon(\bar x;C)$ reduces to the normal cone of $C$ at $\bar x$, which is denoted by $N(\bar x;C)$. However, as a general rule, $	N_\varepsilon(\bar x;C)$ is not a cone when $\varepsilon > 0$.

\medskip
The \textit{polar set} of $A\subset X$ is defined by 
\begin{align*}
A^0=\{ x^*\in X^*\mid \langle x^*, x \rangle \le 1 , \ \forall x \in A \}.
\end{align*}  

\begin{Proposition}\label{proposition_normal_set}{\rm (See \cite[p.~222]{Hiriart_Urruty_1989})}
	The following properties of $\varepsilon$-normal directions are valid: \par {\rm (i)} $N_\varepsilon (x; C)= \varepsilon (C- x)^0$ for any  $x\in C$ and $\varepsilon >0$;\par
	{\rm (ii)} $N(x; C)=\bigcap\limits_{\eta>0} \eta N_\varepsilon (x ; C)$ for any  $x\in C$  and $\varepsilon \ge 0.$
\end{Proposition}

The first assertion of Proposition \ref{proposition_normal_set} shows that the set of the  $\varepsilon$-normal directions $N_\varepsilon (x; C)$ can be computed via the polar set of a set containing 0. Provided that  the set $N_\varepsilon (x; C)$ has been found, by using the second assertion of Proposition \ref{proposition_normal_set}, one can compute the normal cone $N(x; C)$. Due to the importance of the polar sets of sets containing the origin, it is reasonable to consider an illustrative example. Let $X=\mathbb{R}^2$ and $\overline{B}_{\mathbb{R}^2}$ be the unit closed ball in $\mathbb{R}^2.$

\begin{ex}\rm
	Consider the set  $A=\overline{B}((0,1); 1 )=\{ (x_1,x_2)\in \mathbb{R}^2 \mid x_1^2 +(x_2-1)^2\le 1\}$, we have
	$A^0= \{x^*=(x_1^*,x_2^*) \in \mathbb{R}^2 \mid x_1^* + ||x^*|| \le 1 \},$ where $||x^*||=\sqrt{x_1^{*2}+{x_2^{*2}}}$. Indeed, since $A=(1,0)+ \overline{B}_{\mathbb{R}^2},$ we have \begin{align*}
	A^0&=\{x^*=(x_1^*,x_2^*) \in \mathbb{R}^2 \mid \langle (x_1^*,x_2^*), (1,0)+v\rangle \le 1 , \ \forall v \in \overline{B}_{\mathbb{R}^2} \}\\
	&= \{x^* \in \mathbb{R}^2 \mid x_1^* + ||x^*|| \le 1 \}.
	\end{align*}
	\begin{figure}[htbp]
		\centering
		\includegraphics[scale=.40]{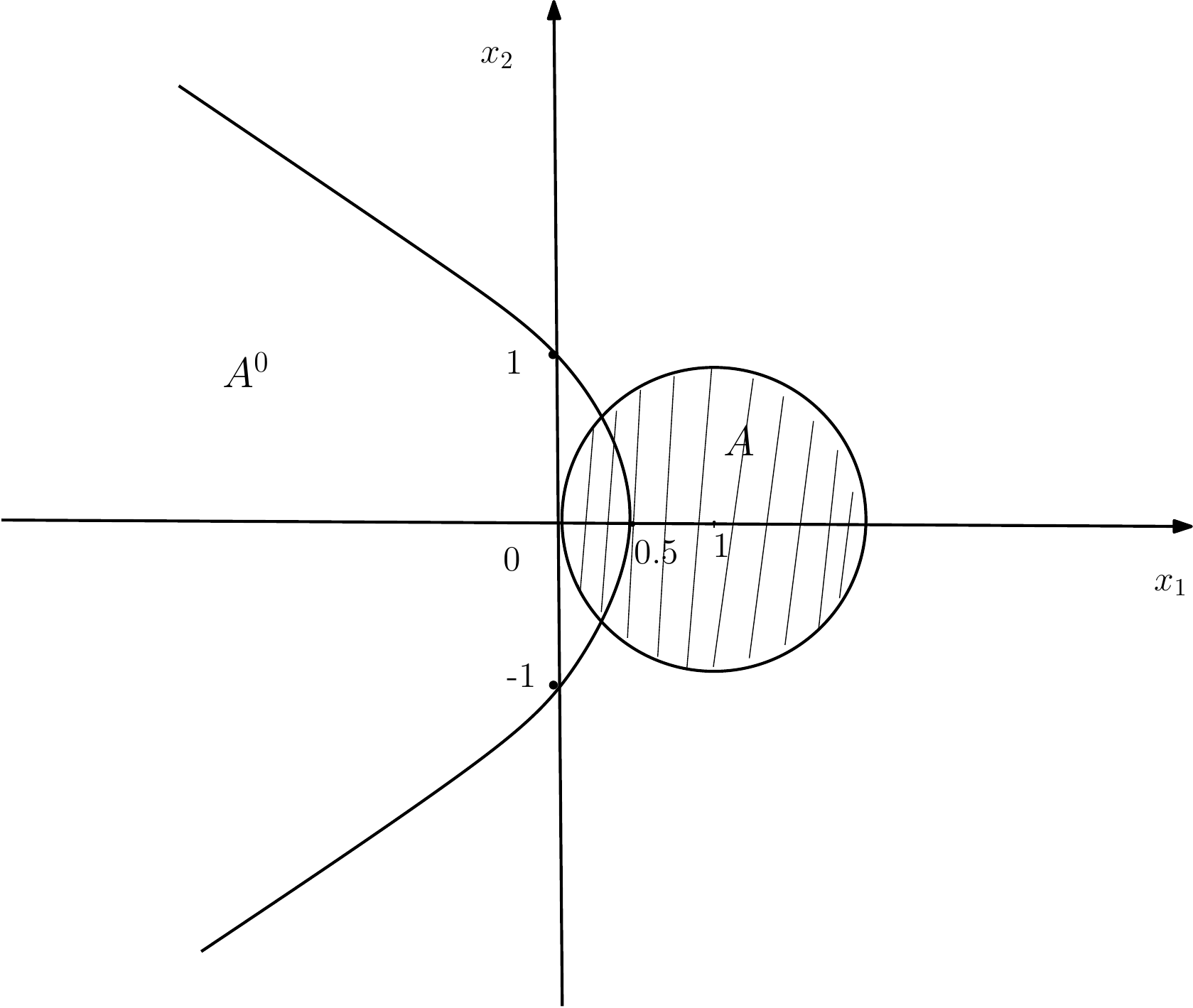}
		\centerline{Figure 1: The polar set of $A$.}
	\end{figure}
\end{ex}

Now, consider a proper convex function $f: X \to \overline{\mathbb{R}}$ and suppose that $\bar x \in {\rm dom}\, f$. The relationship between $\partial_\varepsilon f(\bar x)$ and $N_\varepsilon ((\bar x, f(\bar x));{\rm{epi}}\, f )$ is described  \cite[p.~224]{Hiriart_Urruty_1989} as follows:
\begin{align}\label{varepsilon-normals} \partial_{\varepsilon} f( \bar x)=\big \{x^* \in X^* \mid (x^*,-1)\in N_\varepsilon ((\bar x, f(\bar x));{\rm{epi}}\, f )  \big\}\quad (\varepsilon\geq 0).
\end{align}
Taking $\varepsilon=0$, from \eqref{varepsilon-normals} we recover the following fundamental formula in convex analysis, which relates subdifferentials of a given convex function to the normal cones of its epigraph:
\begin{align*}\partial f(x)=\big \{x^* \in X^* \mid (x^*,-1)\in N((x, f(x));{\rm{epi}}\, f )  \big\}\quad (\forall x\in\dom f).
\end{align*}

\section{Sum rules for $\varepsilon$-subdifferentials}
In convex analysis and optimization, summing two functions is a key operation. The Moreau--Rockafellar Theorem can be viewed as a well-known result, which describes the subdifferential of the sum of two subdifferentiable functions. Invoking a result on the infimal convolution of two functions, one gets a sum rule for $\varepsilon$-subdifferentials.
In the sequel, we will need next fundamental sum rule for $\varepsilon$-subdifferentials.

\begin{Theorem}{\rm (See \cite[Theorem 2.1]{Hiriart_Urruty_1982})}
	\label{sum_rule}
	Suppose that $f_1, f_2: X \to \overline{\Bbb{R}}$ are two proper convex functions on a Hausdorff locally convex topological vector space $X$ and the qualification condition
	\begin{align}
	\label{ConditionH}
	 (f_1\!+\!f_2)^*(x^*)\!=\!\min \big \{ f_1^*(x_1^*)\!+\!f_2^*(x_2^*) \mid x_1^*,\, x_2^*\in X^*,\; x_1^*+x_2^*=x^*\}\ \, (\forall x^*\in X^*)
	\end{align}		
	holds. Then, for every $\bar x \in {\rm dom}\, f_1 \cap {\rm dom}\, f_2$ and $\varepsilon > 0$, one has
	\begin{align}\label{SumRule}
	\partial_\varepsilon (f_1+f_2)(\bar x)=
	\displaystyle\bigcup_{\begin{subarray}{c} \varepsilon_1 \ge 0,\; \varepsilon_2 \ge 0,\\ \varepsilon_1 + \varepsilon_2=\varepsilon \end{subarray}} \big\{  \partial_{\varepsilon_1}f_1(\bar x)+ \partial_{\varepsilon_2}f_2(\bar x)\big\}.
	\end{align}
\end{Theorem}

Condition \eqref{ConditionH} means that, for every $x^*\in X^*$, one has \begin{align}\label{Explain_H}
(f_1+f_2)^*(x^*)=\inf \big \{ f_1^*(x_1^*)+f_2^*(x_2^*) \mid x_1^*,\, x_2^*\in X^*,\, x_1^*+x_2^*=x^*\},\end{align} and the infimum is attained, i.e., there exist $\bar x_1^*,\,\bar x_2^*$ from $X^*$ with $\bar x_1^*+\bar x_2^*=x^*$ such that
\begin{align}\label{Explain_H_1} f_1^*(\bar x_1^*) +f_2^*(\bar x_2^*)=\inf\big \{ f_1^*(x_1^*)+f_2^*(x_2^*) \mid x_1^*+x_2^*=x^*\}.\end{align}

A deeper understanding of condition \eqref{ConditionH} is achieved via the notion of infimal convolution \cite[p.~168]{Ioffe_Tihomirov_1979} of convex functions.

\medskip
The \textit{infimal convolution} $f_1\oplus f_2$ of proper convex functions $f_1:X \rightarrow \Bar{\Bbb{R}}$ and $f_2:X \rightarrow \Bar{\Bbb{R}}$ is defined by
\begin{align*}
(f_1\oplus f_2) (x):=\inf \big \{ f_1(x_1)+f_2(x_2) \mid x_1+x_2=x \}\quad (x\in X).
\end{align*}	
 Applying this construction to the functions $f_1^*:X^* \rightarrow \Bar{\Bbb{R}}$ and $f_2^*:X^* \rightarrow \Bar{\Bbb{R}}$, we have
\begin{align}
\label{infimal_convolution_dual}
(f_1^*\oplus f_2^*) (x^*)=\inf \big \{ f_1^*(x_1^*)+f_2^*(x_2^*) \mid x_1^*+x_2^*=x^*\}.
\end{align}	The attainment of the  infimum on the right-hand-side of \eqref{infimal_convolution_dual} at a point $x^*$ is a kind of \textit{qualification} on the functions $f_1$, $f_2$ in a dual space setting.  The writing $(f_1^*\oplus f_2^*) (x^*)=\min \big \{ f_1^*(x_1^*)+f_2^*(x_2^*) \mid x_1^*+x_2^*=x^*\}$ means that there exist $\bar x_1^*,\, \bar x_2^*$ from $X^*$ with $x^*= \bar x_1^*+\bar x_2^*$ and $(f^*_1\oplus f^*_2) (x^*)=f_1^*(\bar x_1^*) +f_2^*(\bar x_2^*)$.

\medskip
According to \cite[p.~168]{Ioffe_Tihomirov_1979}, the infimal convolution of proper convex functions is a convex function. However, the latter can fail to be proper. For example, if $f_1$ and $f_2$ are linear functions not equal to one another, then their infimal convolution is identically $-\infty$.
\medskip

By the definition of conjugate function, we have
\begin{align*}
(f_1+f_2)^*(x^*)=\sup\limits_{x \,\in\, X} \big\{\langle x^*,x \rangle -(f_1+f_2)(x)   \big \}.
\end{align*}
So, substituting $x^*=x_1^*+x_2^*$ with $x_1^*\in X^*$ and $x_2^*\in X^*$ yields
\begin{align*}
(f_1+f_2)^*(x^*)&=\sup\limits_{x \,\in\, X} \big\{\langle x^*_1+x_2^*,x \rangle -f_1(x)-f_2(x)   \big \}\\
&= \sup\limits_{x \,\in\, X} \big\{\langle x^*_1,x \rangle -f_1(x) + \langle x_2^*, x \rangle -f_2(x)   \big \}\\
&\le \sup\limits_{x \,\in\, X} \big\{\langle x^*_1,x \rangle -f_1(x)\big \} + \sup\limits_{x \,\in\, X}\big\{\langle x_2^*, x \rangle -f_2(x)   \big \}.
\end{align*}
Thus, the inequality
\begin{align}\label{sum_conjugate_function}
(f_1+f_2)^*(x^*)\le f_1^*(x_1^*)+f_2^*(x_2^*)
\end{align}
holds for all $x^*, x^*_1, x^*_2\in X^*$ satisfying $x^*=x_1^*+x_2^*$.  For any $x^* \in X^*$, taking infimum of both sides of \eqref{sum_conjugate_function} on the set of all  $(x^*_1, x^*_2)$ with $x_1^*+x_2^*=x^*$, we get
\begin{align}
\label{analysis_condition}
(f_1+f_2)^*(x^*) \le (f^*_1\oplus f^*_2) (x^*);
\end{align} see \cite[p.~181]{Ioffe_Tihomirov_1979}. Since \eqref{Explain_H} can be rewritten as
\begin{align}\label{Explain_H_2}
(f_1+f_2)^*(x^*) = (f^*_1\oplus f^*_2) (x^*),
\end{align} condition \eqref{ConditionH} requires that, for the functions $f_1$ and $f_2$ in question, the inequality in \eqref{analysis_condition} \textit{holds as equality} for all $x^*\in X^*$. Luckily, this requirement is satisfied under some verifiable regularity conditions. The following theorem describes a condition of this type.

\begin{Theorem}{\rm (See \cite[Theorem 1, p.~178]{Ioffe_Tihomirov_1979})} \label{I_T_theorem}
Suppose that $f_1,f_2$ are proper convex functions. If 
\begin{align}\label{MR_condition}
	\begin{cases}
	one\ of\ the\ functions\ f_1,f_2\ is\ continuous\ at\ a\ point\ belonging\\
	to\ the\ effective\ domain\ of\ the\ other,
	\end{cases} 
\end{align}
then the equality 
$(f_1+f_2)^*(x^*) = (f^*_1\oplus f^*_2) (x^*)$ holds for every $x^*\in X^*$. Moreover, for every $x^* \in {\rm \dom}\, (f_1+f_2)^*$, there exist points $\bar x_i^*\in {\rm \dom}\, f_i^*,\, i=1,2$, such that
$\bar x_1^*+\bar x_2^*=x^*$ and 
\begin{align*}
f_1^*(\bar x_1^*)+f_2^*(\bar x_2^*)=(f_1+f_2)^*(x^*).
\end{align*}
\end{Theorem}
\begin{Remark} \label{Remark1} {\rm 
\textit{Under the assumptions of Theorem \ref{I_T_theorem}, condition \eqref{ConditionH} is satisfied.} Indeed, suppose that one of the proper convex functions $f_1,f_2$ is continuous at a point $x^0$ belonging to the effective domain of the other. Then, one has $x^0\in {\rm{dom}}\, (f_1+f_2)$. It follows that $(f_1+f_2)^*(x^*)$ is everywhere greater than $-\infty$ for all $x^*\in X^*$.
	If  $x^* \notin {\rm \dom}\, (f_1+f_2)^*$, then $(f_1+f_2)^*(x^*)=+\infty$. Choose $\bar x^*_1,\, \bar x^*_2\in X^*$ such that $x^*=\bar x_1^*+\bar x_2^*$. By \eqref{sum_conjugate_function}, $+\infty=(f_1+f_2)^*(x^*)\le f_1^*(\bar x_1^*)+f_2^*(\bar x_2^*).$
Noting that $f_1^*(\bar x_1^*)>-\infty$ and $f_2^*(\bar x_2^*)>-\infty$ because $f_1,f_2$ are proper functions, from this we infer that at least one of the values  $f_1^*(\bar x_1^*)$ and $f_2^*(\bar x_2^*)$ must be $+\infty$. Combining this with \eqref{sum_conjugate_function} yields \eqref{Explain_H_1}. Since \eqref{Explain_H} is equivalent to \eqref{Explain_H_2}, and the latter is fulfilled. Thanks to Theorem \ref{I_T_theorem}, we have thus proved that the equality in \eqref{ConditionH} is satisfied for every $x^* \notin {\rm \dom}\, (f_1+f_2)^*$. If $x^* \in {\rm \dom}\, (f_1+f_2)^*$, then the equality in \eqref{ConditionH} follows immediately from Theorem \ref{I_T_theorem}.} 
\end{Remark}

In a Banach space setting, one has the following analogue of Theorem \ref{I_T_theorem}, where $f_1$ and $f_2$ must be assumed closed. 
Recall that $\Bbb{R_+(}A):=\{ta\in X \mid t \in \Bbb{R}_+ , \, a \in A\}$ and ${\rm int\,A}$, respectively, are the cone generated by the set $A$ and the interior of $A$. 
\begin{Theorem}
	{\rm (See \cite[Theorem 1.1]{AttouchBrezis1986}, \cite[Theorem 4.2 (ii)]{Mordukhovich_Nam_Rector_Tran})}
	\label{A_B_theorem} Let  $f_1,f_2: X \to \overline{\mathbb{R}}$ be proper closed convex functions defined on a Banach space $X$. Suppose that
	\begin{align}\label{AH_condition}
\Bbb{R}_+({\rm \dom}f_1-{\rm \dom}\, f_2)\ is\ a\ nonempty\ closed\ 
subspace\ of\ X. 
	\end{align}
	Then, for every $x^*\in X^*,$ one has 
	$(f_1+f_2)^*(x^*) = (f^*_1\oplus f^*_2) (x^*).$	Moreover, for any
	 $x^* \in {\rm dom}\, (f_1+f_2)^*$ there are $x_1^*,\,x_2^*\in X^*$ such that $x^*=x_1^*+x_2^*$ and $$(f_1+f_2)^*(x^*)=f_1^*(x_1^*)+f_2^*(x_2^*).$$
\end{Theorem}

Later we will also need another version of Theorem \ref{I_T_theorem}, where a geometrical regularity condition is employed.

\begin{Theorem} {\rm (See \cite[Theorem 2.171]{Bonnans_Shapiro_2000})}
 \label{B_S_theorem} Let  $f_1,f_2: X \to \overline{\mathbb{R}}$ be proper closed convex functions defined on a Banach space $X$. If the regularity condition 
\begin{align}\label{interior_condition}
0\in {\rm int}\,({\rm \dom}f_1-{\rm \dom}\, f_2)
\end{align}
is satisfied, then the equality 
$(f_1+f_2)^*(x^*) = (f^*_1\oplus f^*_2) (x^*)$ holds for every $x^*\in X^*$.
Moreover, if $x^*$ is such that the value $(f_1+f_2)^*(x^*)$ is finite, then the set of $ x_1^*$ satisfying $(f_1^* \oplus f_2^*)(x^*)=f_1^*(x_1^*)+f_2^*(x^*-x_1^*)$ is nonempty and weakly$^*$-compact.
\end{Theorem}

 \begin{Remark}\label{Remark2} {\rm  \textit{Under the assumptions of Theorem \ref{A_B_theorem} (resp., of Theorem \ref{B_S_theorem}), condition \eqref{ConditionH} is satisfied.} Indeed, suppose that $f_1,f_2: X \to \overline{\mathbb{R}}$ are proper closed convex functions defined on a Banach space $X$, and \eqref{AH_condition} (resp., \eqref{interior_condition}) is fulfilled. We have $0\in {\rm \dom}f_1-{\rm \dom}\, f_2$. So, there is $x^0\in X$ with $x^0\in {\rm \dom}f_1\cap {\rm \dom}\, f_2$. Then $x^0\in {\rm{dom}}\, (f_1+f_2)$. Applying  Theorem \ref{A_B_theorem} (resp., Theorem \ref{B_S_theorem}) and the arguments already used in Remark \ref{Remark1}, we obtain \eqref{ConditionH}.}
\end{Remark}

We now show that assumption \eqref{ConditionH} is essential for Theorem \ref{sum_rule}.

\begin{ex}\rm
	Let $X=\mathbb{R}$,	$f_1(x)=0 $ for $ x = 0,$ and $f_1(x)=+ \infty$ for $ x\not=0.$  Define $f_2$ by setting $f_2(x)=	- \sqrt{x} $ for $ x \ge 0,$ and $f_2(x)=+ \infty $ for $ x<0.$ By a simple computation we obtain 
	$f^*_1(x^*)=	0$ for all $x^*\in \mathbb{R}$,
	and 
	$$
	f^*_2(x^*)=\begin{cases}
	-\dfrac{1}{4x^*} &\mbox{if}\ \, x^* < 0,\\+ \infty &\mbox{if}\ \, x^*\ge 0.
	\end{cases}
	$$
	Since $(f_1+f_2)(x)=0$ for $x = 0$ and $(f_1+f_2)(x)=+ \infty $ for $x \neq  0$,
	the equality
	$
	(f_1+f_2)^*(x^*)=0$ holds for every $x^*\in \mathbb{R}$.	So, for $x^*=0$, \eqref{Explain_H} holds, but the infimum on the right-hand side is not attained. This means that condition~\eqref{ConditionH} is not satisfied. For $\bar x=0$ and $\varepsilon>0$, the equality \eqref{SumRule} holds because $\partial_\varepsilon (f_1+f_2)(\bar x)=\mathbb R$, 
	$\partial_{\varepsilon_1}f_1(\bar x)=\mathbb R$ for every $\varepsilon_1\geq 0$, $\partial f_2(\bar x)=\emptyset$, and $\partial_{\varepsilon_2}f_2(\bar x)=\left( -\infty,\, -\frac{1}{4\varepsilon_2}\right]$ for every $\varepsilon_2>0$ (see Example \ref{Ex_sqrt_minus}). Nevertheless, for $\bar x=0$ and $\varepsilon=0$, the equality \eqref{SumRule} is violated because the left-hand side is $\mathbb R$, while the right-hand side is the empty set.
\end{ex}

 The sum rule \eqref{SumRule} requires the fulfillment of condition \eqref{ConditionH}, which is implied by the regularity conditions \eqref{MR_condition}, \eqref{AH_condition}, and \eqref{interior_condition} and the corresponding assumptions of Theorems \ref{I_T_theorem}, \ref{A_B_theorem}, and  \ref{B_S_theorem}. We now clarify the relationships between the regularity conditions  \eqref{MR_condition}, \eqref{AH_condition}, and \eqref{interior_condition}.

		\begin{Proposition} {\rm (See also \cite[Subsection 6.1]{AnYen2015})} Let $f_1,f_2: X \to \overline{\mathbb{R}}$ be proper closed convex functions defined on  a Hausdorff locally convex topological vector space $X$. Then,  \eqref{MR_condition} implies \eqref{AH_condition} and \eqref{interior_condition}.
		\end{Proposition}
	\begin{proof}
	Without loss of generality, suppose that $f_1$ is continuous at a point $\bar x\in {\rm \dom}\, f_2$. Then, there exists a neighborhood $U$ of $0\in X$ such that $\bar x +U\subset {\rm \dom}\, f_1$. So, $U=(\bar x +U) - \bar x \subset {\rm \dom}\, f_1 -{\rm \dom}\, f_2.$ This yields \eqref{interior_condition} and the equality $$\Bbb{R}_+({\rm \dom}f_1-{\rm \dom}\, f_2)= X,$$ which justifies  \eqref{AH_condition}.
	\end{proof}
	
	The implication \eqref{interior_condition} $\Rightarrow$ \eqref{AH_condition} is obvious. Let us present two simple examples to show that the converse implication and the assertion \eqref{interior_condition} $\Rightarrow$ \eqref{MR_condition} are not true.

\begin{ex}\label{Example 3.2} {\rm
Let $X=\Bbb{R}^2$, $f_1(x)=x_1^2$ for all $x=(x_1,0)$, $f_1(x)=+\infty$ for all $x=(x_1,x_2)$ with $x_1\neq 0$, and $f_2\equiv f_1$. Then, $$\Bbb{R}_+({\rm \dom}f_1-{\rm \dom}\, f_2)={\rm \dom}f_1-{\rm \dom}\, f_2=\Bbb{R}\times\{0\}$$ is a closed subspace of $X$. However, both conditions \eqref{MR_condition} and \eqref{interior_condition} are violated.}
\end{ex}

\begin{ex}{\rm Let $X$ and $f_1$ be the same as in Example \ref{Example 3.2}. Put $f_2(x)=x_2^2$ for all $x=(0,x_2)$, $f_2(x)=+\infty$ for all $x=(x_1,x_2)$ with $x_2\neq 0$. Then  \eqref{interior_condition} is satisfied, but~\eqref{MR_condition} fails to hold.		
}\end{ex}

\section{Main results}
		\markboth{\centerline{\it Main Results}}{\centerline{\it D.T.V.~An
		and J.-C.~Yao}} \setcounter{equation}{0}
	
Differential stability of convex optimization problems with possibly empty solution sets in infinite-dimensional spaces is studied in this section. To make the presentation as clear as possible, we distinguish two cases:
 
a) unconstrained problems;

b) constrained problems.

\medskip
Let $X,Y$ be Hausdorff locally convex topological vector spaces and $\varphi: X \times Y \rightarrow \overline{\mathbb{R}}$ an extended real-valued function. 

		\subsection{Unconstrained convex optimization problems}
Consider the \textit{parametric unconstrained convex optimization problem}
					\begin{align}\label{math_program}
				\min\{\varphi(x,y)\mid y \in Y\}
				\end{align}
			 depending on the parameter $x$. The function $\varphi$ is called the \textit{objective function} of~\eqref{math_program}. The \textit{optimal value function}
			 $\mu: X \rightarrow \overline{\mathbb{R}}$ of \eqref{math_program} is
			 \begin{align}\label{marginalfunction}
			 \mu(x):= \inf \left\{\varphi (x,y)\mid y \in Y\right\}.
			 \end{align}
			 The \textit{solution set} of \eqref{math_program} is defined by
			$M(\bar x):=\{y \in Y\mid \mu(\bar x)= \varphi (\bar x,y)\}.$
	For $\eta > 0$, one calls $M_\eta(\bar x):= \{ y \in Y \mid \varphi (\bar x, y) \le \mu (\bar x)+\eta \}$
the \textit{approximate solution set} of \eqref{math_program}.

\medskip
		We now obtain formulas for the $\varepsilon$-subdifferential of $\mu(.)$. Since the following result was given in  \cite[Corollary~5]{MoussaouiSeeger1994} as a consequence of a more general result and in \cite[Theorem~2.6.2]{Zanlinescu_2002} with a brief proof, we will present a detailed, direct proof to make the presentation as clear as possible. Our arguments are based on a proof scheme of \cite{MoussaouiSeeger1994}. 
		
		\begin{Theorem}\label{Corolary5} {\rm (See  \cite[Corollary~5]{MoussaouiSeeger1994} and \cite[Theorem 2.6.2, p.~109]{Zanlinescu_2002})} Suppose that $\varphi: X \times Y \to \overline {\mathbb{R}}$ is a proper convex function and $\mu(\cdot)$ is finite at $\bar x \in X.$ Then, for every $\varepsilon \ge 0$, one has
		\begin{equation}\label{Epsilon_subdifferential2}
		\begin{array}{rcl}\partial_\varepsilon \mu(\bar x) &=& \bigcap\limits_{\eta\; >\;0} \ \bigcap\limits_{y \,\in\, M_\eta (\bar x)} \bigg\{x^* \in X^* \mid (x^*,0) \in \partial_{\varepsilon +\eta} \varphi (\bar x, y) \bigg\}\\
		&=&\bigcap\limits_{\eta\; >\;0} \ \bigcup\limits_{y \,\in\, Y} \bigg\{x^* \in X^* \mid (x^*,0) \in \partial_{\varepsilon +\eta} \varphi (\bar x, y) \bigg\}. 
		\end{array}
		\end{equation}
	In particular, 
	\begin{equation}\label{subdifferential}
		\begin{array}{rcl}	\partial \mu(\bar x) &=& \bigcap\limits_{\eta\; >\;0} \ \bigcap\limits_{y \,\in\, M_\eta (\bar x)} \bigg\{x^* \in X^* \mid (x^*,0) \in \partial_{\eta} \varphi (\bar x, y) \bigg\}\\
	&=&\bigcap\limits_{\eta\; >\;0} \ \bigcup\limits_{y \,\in\, Y} \bigg\{x^* \in X^* \mid (x^*,0) \in \partial_{\eta} \varphi (\bar x, y) \bigg\}.
	\end{array}
	\end{equation}
Moreover, if $M(\bar x)\not= \emptyset$, then for every $\varepsilon \ge 0$, one has
		\begin{align}\label{nonconstraint}
		\partial_\varepsilon \mu(\bar x) = \big\{x^* \in X^* \mid (x^*,0) \in \partial_{\varepsilon} \varphi (\bar x, y) \big\},
		\end{align}
		for all $y \in M(\bar x).$
		\end{Theorem}
		\begin{proof}
 We put
\begin{align*}
&\mathcal{M}_\eta (\bar x)= \bigcap\limits_{y \,\in\, M_\eta (\bar x)} \bigg\{x^* \in X^* \mid (x^*,0) \in \partial _{\varepsilon +\eta} \varphi (\bar x, y)
 \bigg\},\\
 &\mathcal{N}_\eta (\bar x)= \bigcup\limits_{y \,\in\, Y} \bigg\{x^* \in X^* \mid (x^*,0) \in \partial _{\varepsilon +\eta} \varphi (\bar x, y)
  \bigg\}.
\end{align*}
Since $\mu(\bar x)=\inf\limits_{y\, \in\, Y} \varphi (\bar x, y)$ by \eqref{marginalfunction}, the set $M_\eta(\bar x)$ is nonempty for every $\eta >0$. Thus, one has $\mathcal{M}_\eta (\bar x)\subset \mathcal{N}_\eta (\bar x)$ for all $\eta >0$. Hence $\bigcap\limits_{\eta>0}\mathcal{M}_\eta (\bar x)\subset \bigcap\limits_{\eta>0} \mathcal{N}_\eta (\bar x)$. So, the equalities in \eqref{Epsilon_subdifferential2} will be proved, if we can show that
\begin{equation}
\label{ct1}
\partial_\varepsilon \mu(\bar x) \subset \bigcap\limits_{\eta\,>\,0}\mathcal{M}_\eta (\bar x)\end{equation} and
\begin{equation}\label{ct1n}
\bigcap\limits_{\eta\,>\,0} \mathcal{N}_\eta (\bar x) \subset \partial_\varepsilon \mu(\bar x).
\end{equation}
To prove \eqref{ct1}, take any $x^* \in \partial_\varepsilon \mu(\bar x)$, $\eta>0$, and $y \in M_\eta(\bar x)$. Thanks to the first assertion of Proposition~\ref{Proposition_1}, we know that $x^* \in \partial_\varepsilon \mu(\bar x)$ if and only if
\begin{align}
\label{ct2}
\mu(\bar x) + \mu^*(x^*) \le \langle x^*, \bar x \rangle +\varepsilon.
\end{align}
Adding $\eta>0$ to both sides of \eqref{ct2} yields
\begin{align}
\label{ct3}
\mu(\bar x) + \mu^*(x^*) + \eta\le \langle x^*, \bar x \rangle+\varepsilon +\eta.
\end{align}
Since $y \in M_\eta (\bar x)$, one has $\varphi(\bar x, y) \le \mu(\bar x)+\eta$. So, \eqref{ct3} gives
\begin{align}
\label{ct4}
\varphi(\bar x,y) + \mu^*(x^*) \le \langle x^*, \bar x \rangle +\varepsilon +\eta.
\end{align}
For every $v^* \in X^*$, we have $\mu^*(v^*)=\varphi^*(v^*,0).$ Indeed, by the definition of conjugate function, 
\begin{align*}
\mu^*(v^*)&= \sup\limits_{x \,\in\, X} \big \{ \langle v^*, x \rangle - \mu (x)\big \}\\ & = \sup\limits_{x \in X} \big \{ \langle v^*, x \rangle - \inf\limits_{y \,\in\, Y} \varphi (x,y)\big \}\\
&= \sup\limits_{(x,y) \,\in\, X\times Y} \big \{ \langle v^*, x \rangle - \varphi (x,y)\big \}\\
&= \sup\limits_{(x,y) \,\in \,X\times Y} \big \{ \langle (v^*,0), (x,y) \rangle - \varphi (x,y)\big \}\\
&= \varphi^*(v^*,0).
\end{align*}
Substituting $\mu^*(x^*)=\varphi^*(x^*,0)$	into \eqref{ct4}, one obtains
\begin{align}
\label{ct5}
\varphi(\bar x,y) + \varphi^*(x^*,0) \le \langle x^*, \bar x \rangle +\varepsilon +\eta.
\end{align}
According to Proposition \ref{Proposition_1},  inequality~\eqref{ct5} yields $(x^*,0)\in \partial_{\varepsilon +\eta} \varphi (\bar x, y)$ for all $\eta>0$ and $y \in M_\eta(\bar x).$ This means that $x^* \in \bigcap\limits_{\eta\,>\,0}\mathcal{M}_\eta (\bar x)$, so \eqref{ct1} is valid.

 Next, to prove \eqref{ct1n}, take any $x^* \in \bigcap\limits_{\eta>0}\mathcal{N}_\eta (\bar x)$. Then, for every $\eta>0$, there exists $y \in Y$ such that $(x^*,0)\in \partial_{\varepsilon +\eta} \varphi(\bar x, y)$. By Proposition \ref{Proposition_1}, this means that 
 $
 \varphi^*(x^*,0) +\varphi (\bar x,y) - \langle (x^*,0), (\bar x, y) \rangle \le \varepsilon +\eta.
 $
 The latter yields
 \begin{align}\label{ct6}
  \varphi^*(x^*,0) +\varphi (\bar x,y) - \langle x^*,\bar x \rangle \le \varepsilon +\eta.
  \end{align}
Since $ \varphi^*(x^*,0)=\mu^*(x^*)$ and $ \mu(\bar x) \le \varphi (\bar x, y)$, \eqref{ct6} implies
\begin{align}\label{ct7}
  \mu^*(x^*) +\mu (\bar x) - \langle x^*, \bar x \rangle \le \varepsilon +\eta.
  \end{align}  
  As \eqref{ct7} holds for every $\eta >0$, letting $\eta \to 0^+$ yields
  $ \mu^*(x^*) +\mu (\bar x) - \langle x^*, \bar x \rangle \le \varepsilon .$
  The last inequality shows that $x^* \in \partial_\varepsilon \mu(\bar x)$. Therefore, \eqref{ct1n} is fulfilled.
  
  Combining \eqref{ct1} and \eqref{ct1n} gives \eqref{Epsilon_subdifferential2}. For $\varepsilon=0$, from \eqref{Epsilon_subdifferential2} one obtains \eqref{subdifferential}.
	\end{proof}

 Next elementary property of the $\varepsilon$-subdifferential will be used latter on.		
 
\begin{Proposition}\label{times_subdifferential}
	Let $\varphi: X \times Y \to \mathbb{R}$ be a convex function. If $\varphi(x,y)= \varphi_1(x)+\varphi_2(y)$, where $\varphi_1: X \to \mathbb{R}$ and $\varphi_2: Y \to \mathbb{R}$ are convex functions then, 	for any $\varepsilon \ge 0$ and $(\bar x, \bar y) \in X \times Y$, one has
	\begin{align}\label{rule1}
	\partial_{\varepsilon} \varphi (\bar x, \bar y)\subset \partial_{\varepsilon}\varphi_1 (\bar x) \times\partial_{\varepsilon} \varphi_2 (\bar y)\subset \partial_{2\varepsilon} \varphi (\bar x, \bar y).
	\end{align}
\end{Proposition}
\begin{proof}
	Suppose that $(x^*, y^*) \in \partial_{\varepsilon} \varphi (\bar x, \bar y)$ for some $\varepsilon \ge 0$. Then, we have	
	\begin{align}
	\label{reprentation1}
	\langle (x^*,y^*), (x,y)-(\bar x, \bar y) \rangle \le \varphi (x,y) -\varphi(\bar x, \bar y)+\varepsilon, \ \forall (x,y)\in X \times Y.
	\end{align}
	By our assumption, \eqref{reprentation1} is equivalent to
	\begin{align}
	\label{reprentation2}
	\langle x^*, x-\bar x \rangle + \langle y^*, y -\bar y \rangle \le \varphi_1 (x) -\varphi_1(\bar x) +\varphi_2(y)-\varphi_2(\bar y)+\varepsilon, \ \forall (x,y)\in X \times Y.
	\end{align}
On one hand, substituting $y =\bar y$ into \eqref{reprentation2}, we get $x^* \in \partial\varphi_1(\bar x)$. On the other hand, taking $x =\bar x$, from \eqref{reprentation2} we have $y^* \in \partial\varphi_2(\bar y).$ Therefore, for any $ \varepsilon \ge 0$,
	$$ \partial_{\varepsilon} \varphi (\bar x, \bar y)\subset  \partial_{\varepsilon}\varphi_1 (\bar x) \times\partial_{\varepsilon} \varphi_2 (\bar y).$$
	The second inclusion in \eqref{rule1} can be obtained easily by the definition of $\varepsilon$-subdifferential. Thus \eqref{rule1} is valid.
\end{proof}

The following example is taken from \cite[pp.~93--94]{Hiriart_Lemarechal_1993}.

\begin{ex}\label{norm_function}\rm Let $f(x)=|x|$ for all $x\in\mathbb{R}$ and $\varepsilon\geq 0$. We have 
	\begin{align*}
	\partial _\varepsilon f(x)=
	\begin{cases}
	\left[-1,\, -1 -\dfrac{\varepsilon}{x}  \right] & \mbox{if} \ \,  x\,<\, \dfrac{- \varepsilon}{2},\\
	[-1, 1]& \mbox{if} \ \, \dfrac{- \varepsilon}{2} \,\le\, x \le \dfrac{\varepsilon}{2},\\
	\left[1-\dfrac{\varepsilon}{x}, \, 1 \right]& \mbox{if} \ \, x\,>\,\dfrac{ \varepsilon}{2}.
	\end{cases}
	\end{align*}
\end{ex}

We now give an illustration for Theorem~\ref{Corolary5}.

\begin{ex}\rm 
	Choose $X=Y=\mathbb{R}$, $\varphi(x,y)=x^2 +|y|$, and $\bar x=0$. Then the optimal value function \eqref{marginalfunction} of the parametric problem \eqref{math_program} is  $\mu(x)=x^2$. For any $\varepsilon \ge 0$, we have 
	\begin{align*}
	\partial_\varepsilon \mu (\bar x)=& \{x^*\in \mathbb{R}\mid x^* x \le x^2 + \varepsilon, \ \forall x \in \mathbb{R} \}\\
	= &\{ x^* \in \mathbb{R} \mid -x^2 +x^*x -\varepsilon \le 0, 
	\ \forall x\in \mathbb{R} \}\\
	=&\left[ -2 \sqrt{\varepsilon},\ 2 \sqrt{\varepsilon}\right].
	\end{align*}
	In this case, $\bar y =0 \in M(\bar x)$, so we will clarify equality \eqref{nonconstraint}. By Proposition \ref{times_subdifferential} one has $\partial_{\varepsilon} \varphi (\bar x,y)\subset\partial_{\varepsilon} \varphi_1 (\bar x) \times\partial_{\varepsilon } \varphi_2(y),$  where $\varphi_1(x)=x^2$ and $\varphi_2(y)=|y|$. On one hand, $\partial_{\varepsilon } \varphi_1(\bar x)=\left[ -2 \sqrt{\varepsilon},\ 2 \sqrt{\varepsilon} \right]$. On the other hand, according to Example \ref{norm_function},
	\begin{align*}
	\partial _{\varepsilon} \varphi_2(y)=
	\begin{cases}
	\left[-1, \ \, -1 -\dfrac{\varepsilon}{y}  \right] & \mbox{if} \ y\ <\ -\dfrac{ \varepsilon}{2},\\
	[-1,\ \, 1]& \mbox{if} \ -\dfrac{ \varepsilon}{2}\ \le\ y\ \le\ \dfrac{\varepsilon}{2},\\
	\left[1-\dfrac{\varepsilon}{y},\ \,  1 \right]& \mbox{if} \ y\ >\ \dfrac{ \varepsilon}{2}.
	\end{cases}
	\end{align*}
	Then, the right-hand side of \eqref{nonconstraint} can be computed as follows
	\begin{align*}
	RHS_{\eqref{nonconstraint}}
	= & \bigg\{ x^* \in \mathbb{R} \mid (x^*,0)\in \left[ -2 \sqrt{\varepsilon},\ 2 \sqrt{\varepsilon} \right] \times [-1, 1]
	\bigg\}\\
	=& \left[ -2 \sqrt{\varepsilon},\ 2 \sqrt{\varepsilon}\right].
	\end{align*}
Therefore, the conclusion of Theorem~\ref{Corolary5} is justified.
\end{ex}
\subsection{Constrained convex optimization problems}\label{main_subsection}
	Let $\varphi: X \times Y \rightarrow \overline{\mathbb{R}}$ be an extended real-valued funtion, $G: X \rightrightarrows Y$ a multifunction between Hausdorff locally convex topological vector spaces.	Consider the \textit{parametric optimization problem under an inclusion constraint}
			\begin{align}\label{math_program1}
		\min\{\varphi(x,y)\mid y \in G(x)\}
		\end{align}
	 depending on the parameter $x$. The function $\varphi$ (resp., the multifunction $G$) is called the \textit{objective function} (resp., the \textit{constraint multifunction}) of \eqref{math_program1}. The \textit{optimal value function}
	 $\mu: X \rightarrow \overline{\mathbb{R}}$ of \eqref{math_program1} is
	 \begin{align}\label{marginalfunction1}
	 \mu(x):= \inf \left\{\varphi (x,y)\mid y \in G(x)\right\}.
	 \end{align}
	 The usual convention $\inf \emptyset =+\infty$ forces $\mu(x)=+\infty$ for every $x \notin {\rm{dom}}\, G.$ 
	 The \textit{solution map}  $M: {\rm {dom}}\, G  \rightrightarrows Y $ of \eqref{math_program1} is defined by
		\begin{align*}
		M(x)=\{y \in G(x)\mid \mu(x)= \varphi (x,y)\}.
		\end{align*}
		The \textit{approximate solution set} of \eqref{math_program1} is given by 
		\begin{align}\label{solutionset}
		M_\eta(\bar x)= \{ y \in G(\bar x) \mid \varphi (\bar x, y) \le \mu (\bar x)+\eta \},\ \, \forall \eta > 0.
		\end{align} 
		
		\medskip	

We are now in a position to formulate the first main result of this subsection.  For any $\varepsilon\ge 0$ and $ \eta \ge 0$, define by  $\varGamma(\eta+\varepsilon)$ the set $$\varGamma(\eta+\varepsilon)=\{(\gamma_1, \gamma_2) \mid \gamma_1 \ge 0,\ \gamma_2 \ge 0, \ \gamma_1+\gamma_2=\eta+\varepsilon\}.$$
\begin{Theorem}
\label{main_result1} Let $\varphi: X \times Y \to \overline {\mathbb{R}}$ be a proper convex function, $G: X \rightrightarrows Y$ a convex multifunction. Suppose that the optimal value function $\mu(\cdot)$ in \eqref{marginalfunction1} is finite at $\bar x \in X.$ If at least one of the following regularity conditions is satisfied:\\
{\rm (a)} ${\rm int(\gph}\, G)\cap {\rm \dom}\, \varphi \not =\emptyset,$\\
{\rm (b)} $\varphi$ is continuous at a point $(x^0,y^0)\in {\rm gph}\, G,$
\\
then, for every $\varepsilon \ge 0$, we have
\begin{equation}\label{Main_formula}
\begin{split}
&\partial _\varepsilon \mu(\bar x)\\
&=
\bigcap\limits_{\eta >0} 
\ \bigcap\limits_{ y \,\in\, M_\eta (\bar x)}\
\bigcup_{(\gamma_1, \gamma_2)\, \in \, \varGamma(\eta+\varepsilon)}
\bigg\{x^* \mid (x^*,0) \in \partial_{\gamma_1}\varphi(\bar x,y) \!+\!N_{\gamma_2}\big ((\bar x,  y); {\rm gph}\, G \big ) \bigg\}\\
&=\bigcap\limits_{\eta >0} 
\ \bigcup\limits_{y\, \in \, Y}\
\bigcup_{(\gamma_1, \gamma_2) \,\in\, \varGamma(\eta+\varepsilon)}
\bigg\{x^* \mid (x^*,0) \in \partial_{\gamma_1}\varphi(\bar x, y) +N_{\gamma_2}\big ((\bar x, y); {\rm gph}\, G \big ) \bigg\},
  \end{split}
    \end{equation}
 where $M_\eta (\bar x)$ is given in \eqref{solutionset}.
 \end{Theorem}	
\begin{proof} (This proof is based on Theorems \ref{sum_rule} and \ref{Corolary5}.)
 We apply Theorem \ref{Corolary5} to the case where $\varphi (x,y)$ plays the role of $\big (\varphi +\delta (\cdot; {\rm \gph}\, G)\big )(x,y)$. Hence
\begin{equation}\label{N_Formula_1}
\begin{split}
		\partial_\varepsilon \mu(\bar x)& = \bigcap\limits_{\eta >0} 
		\ \bigcap\limits_{ y \,\in\, M_\eta (\bar x)}\bigg\{x^* \in X^* \mid (x^*,0) \in \partial_{\varepsilon +\eta} \big (\varphi +\delta (\cdot; {\rm \gph}\, G)\big ) (\bar x,  y) \bigg\}\\
	&	=\bigcap\limits_{\eta >0} \ \bigcup\limits_{  y\, \in \, Y} \bigg\{x^* \in X^* \mid (x^*,0) \in \partial_{\varepsilon +\eta} \big (\varphi +\delta (\cdot; {\rm \gph}\, G)\big ) (\bar x,  y) \bigg\}.
		\end{split}
		\end{equation}
We will show that 
\begin{align}\label{N_Formula_2}
\partial_{\varepsilon +\eta}\big (\varphi +\delta (\cdot; {\rm \gph}\, G)\big ) (\bar x,  y)\!= \! \bigcup_{(\gamma_1, \gamma_2) \,\in\, \varGamma(\eta+ \varepsilon)}
  \bigg\{ \partial_{\gamma_1}\varphi(\bar x, y)\! +\!N_{\gamma_2}((\bar x, y); {\rm gph}\, G)  \bigg\},
\end{align}		
 where $ \varGamma(\eta+ \varepsilon)=\{(\gamma_1, \gamma_2) \mid \gamma_1 \ge 0,\ \gamma_2 \ge 0,\ \gamma_1+\gamma_2=\varepsilon+\eta\}.$  
 Indeed, suppose that at least one of the regularity conditions (a) or (b) is fulfilled. Since ${\rm \gph}\,G$ is convex, $\delta(\cdot; {\rm \gph}\,G): X \times Y \to \overline {\mathbb{R}}$ is convex. Obviously, $\delta(\cdot; {\rm \gph}\,G)$ is continuous at every point belonging to ${\rm int (\gph\, G)}$. Hence, if the regularity condition (a) is satisfied, then $\delta(\cdot; {\rm \gph}\,G)$ is continuous at a point in ${\rm \dom}\, \varphi$. Consider the case where the regularity condition (b) is fulfilled. Since ${\rm \dom}\delta(\cdot; {\rm \gph}\,G)={\rm \gph}\,G $. From (b), it follows that $\varphi$ is continuous at a point in ${\rm \dom}\delta(\cdot; {\rm \gph}\,G)$. So, in both cases, thanks to Theorem~\ref{I_T_theorem} and Remark \ref{Remark1}, the qualification condition \begin{align}
 \label{RegularityCondition}
 \nonumber
 \big(\varphi &+\delta(\cdot; {\rm \gph}\,G)\big)^*(x^*,y^*) \\
 & =\min \bigg\{ \varphi ^*(x_1^*,y_1^*) +\delta^*((x_2^*,y_2^*);{\rm \gph}\,G)\mid(x^*,y^*)=(x_1^*,y_1^*)+(x_2^*,y_2^*)  \bigg \}
 \end{align}
 holds for all $(x^*,y^*)\in X^*\times Y^*$. So, all assumptions of Theorem \ref{sum_rule} are satisfied. Therefore,
\begin{displaymath}
 \partial_{\varepsilon+\eta} \big (\varphi +\delta (\cdot; {\rm \gph}\, G)\big )(\bar x, y)=
 \bigcup_{(\gamma_1, \gamma_2) \,\in\, \varGamma(\eta+ \varepsilon)} \bigg \{  \partial_{\gamma_1}\varphi(\bar x,  y)+ \partial_{\gamma_2}\delta((\bar x, y); {\rm \gph}\, G)\bigg \},
 \end{displaymath}
 for any $(\bar x, y) \,\in\, {\rm dom}\, \varphi \cap\, {\rm\gph}\, G$. Moreover, $\partial_{\gamma_2}\delta\big ((\bar x, y); {\rm \gph}\, G\big )= N_{\gamma_2}\big ((\bar x, y); {\rm gph}\, G \big ).$ Combining \eqref{N_Formula_1} with \eqref{N_Formula_2}, we obtain the statement of the theorem.
\end{proof}

The second main result of this section reads as follows.

\begin{Theorem}
Let $G :X\rightrightarrows Y$ be a convex multifunction between Banach spaces, whose graph is closed, and $\varphi: X \times Y \to \overline {\mathbb{R}}$ a proper closed convex  function. Suppose that the optimal value function $\mu(\cdot)$ in \eqref{marginalfunction1} is finite at $\bar x \in X.$ Assume that either 
{\rm (i)} 
${\Bbb{R_+}}( {\rm \dom}\, \varphi- {\rm \gph}\, G)$ is a closed subspace of $X \times Y$,

\noindent or

\noindent {\rm (ii)}  $
(0,0)\in {\rm int}( {\rm \dom}\, \varphi- {\rm \gph}\, G),
$ 

\noindent then \eqref{Main_formula} is valid.
\end{Theorem}
\begin{proof}  
	The proof is similar to that of Theorem \ref{main_result1}. Having in hands the subdifferential representation for the optimal value function in Theorem \ref{Corolary5}, we apply therein the subdifferential sum rule from Theorem \ref{sum_rule} under the corresponding conditions~(i) and~(ii). Namely, if the condition (i) (resp. (ii)) is satisfied, using Theorem~\ref{A_B_theorem} (resp. Theorem~\ref{B_S_theorem}) and remembering Remark~\ref{Remark2}, then we obtain \eqref{RegularityCondition}. In other words, all assumptions of Theorem \ref{sum_rule} are satisfied.
Thus, by the same manner as in Theorem~\ref{main_result1}, we can obtain the conclusion of the theorem.
\end{proof}

\subsection{An application}
In this section, we will present an illustrative example for the result in Subsection~\ref{main_subsection}. This example is designed for the case graph of the constraint mapping is a convex cone.
\medskip

 We have the following property about $\varepsilon$-normal directions of a convex cone.

\begin{Proposition}{\rm (See \cite[Example 2.1]{Hiriart_Urruty_1989})}\label{Property_normal_cone} Let $C$ be a convex cone with apex 0. Then one has for all $\bar x \in C$ and all $\varepsilon \ge 0$ the equality
	\begin{align*}
		N_\varepsilon (\bar x; C)= \{x^* \in C^0 \mid \langle x^*, \bar x \rangle \ge - \varepsilon \}.
	\end{align*}
	In particular, 
	$N_\varepsilon (\bar x; C)=N(\bar x; C)$ for $\bar x=0.$
\end{Proposition}
\begin{proof}
	For all $\varepsilon \ge 0$, take any $x^* \in	N_\varepsilon (\bar x; C)$. By the definition of $\varepsilon$-normal directions, we have
	\begin{align}\label{normal_cone}
	\langle x^*, x- \bar x \rangle \le \varepsilon, \ \, \forall x \in C.
	\end{align}
	Substituting $x=0$, we get $\langle x^*, \bar x \rangle \ge - \varepsilon$. Moreover, since $C$ is a convex cone, $\bar x + ty \in C$, for all $t >0$, $y \in C$. Now taking $x=\bar x + ty$, \eqref{normal_cone} yields
	\begin{align}\label{normal_cone1}
	t \langle x^*, y \rangle \le \varepsilon, \ \, \forall y \in C.
	\end{align}
	Dividing two sides of \eqref{normal_cone1} by $t>0$ and letting $t \to +\infty$, we obtain $\langle x^*, y \rangle \le 0,$ for all $y \in C$. The latter means that $x^* \in C^0$.
	
	Now suppose that $x^* \in C^0$ and $\langle x^*, \bar x \rangle \ge - \varepsilon $ for every $\varepsilon \ge 0$. Given any $x \in C$, we have $\langle x^*, x \rangle \le 0$. Combining this with $\langle x^*, \bar x \rangle \ge - \varepsilon $, we obtain $x^* \in N_\varepsilon (\bar x; C).$
	\end{proof}

\medskip
We can easily get the following property of $\varepsilon$-subdifferentials.
\begin{Proposition}\label{properties} 
	Let $f: X \rightarrow \Bar{\Bbb{R}}$ be a proper convex function. Then, for any $\varepsilon \ge 0$ and $\bar x \in {\rm{dom}}\,f$ we have
 $\partial_\varepsilon (\lambda f)(\bar x)=\lambda \partial_{\varepsilon/\lambda} f(\bar x)$ for every $\lambda >0.$
\end{Proposition}
Let us consider an illustrative example for Theorem~\ref{main_result1}.
\begin{ex}{\rm
		Let $X=Y=\Bbb{R}$ and $\bar x=0$. Consider the optimal value function $\mu(x)$ in~\eqref{marginalfunction1}
		with $ \varphi(x,y)=|y|$ and $G(x)=\left\{y \mid y \ge \frac{1}{2}|x|\right\}$ for all $x\in\Bbb{R}$. Then we have $\mu(x)=\frac{1}{2}|x|$ for all $x\in\Bbb{R}$. From Example~\ref{norm_function} and Proposition~\ref{properties}, for any $\varepsilon \ge 0$ one has $\partial_\varepsilon \mu( \bar x)=[-\frac{1}{2},\frac{1}{2}]$. On one hand, for all $\eta >0$, $\gamma_1 \ge 0$, we get
		 \begin{align*}
		M_\eta(\bar x)=\{y \in G(\bar x) \mid \varphi(\bar x, y) \le \mu(\bar x) + \eta \}=\{0\},
		 \end{align*}
and 
\begin{align*}
\partial_{\gamma_1}\varphi(\bar x, y)\subset
\begin{cases}
\{0\}\times \left[-1, \ \, -1 -\dfrac{\gamma_1 }{y}  \right] & \mbox{if} \ y\ <\ -\dfrac{ \gamma_1 }{2},\\
\{0\}\times [-1,\ \, 1]& \mbox{if} \ -\dfrac{ \gamma_1 }{2}\ \le\ y\ \le\ \dfrac{\gamma_1 }{2},\\
\{0\}\times \left[1-\dfrac{\gamma_1 }{y},\ \,  1 \right]& \mbox{if} \ y\ >\ \dfrac{ \gamma_1 }{2}.
\end{cases}
\end{align*}
On the other hand, by Proposition~\ref{Property_normal_cone}, we have 
\begin{align*}
N_{\gamma_2} ((\bar x, y); {\rm gph}\,G )= \begin{cases}
\{(0,0)\} & \mbox{if}\ \,  y>0,\\
\left\{(x^*,y^*)\in \mathbb{R}^2 \mid y^* \le - 2|x^*| \right\}& \mbox{if}\ \, y=0,\\
\emptyset & \mbox{if}\ \, y <0.
\end{cases}
\end{align*}
 \begin{figure}[htbp]
	\centering
	\includegraphics[scale=.55]{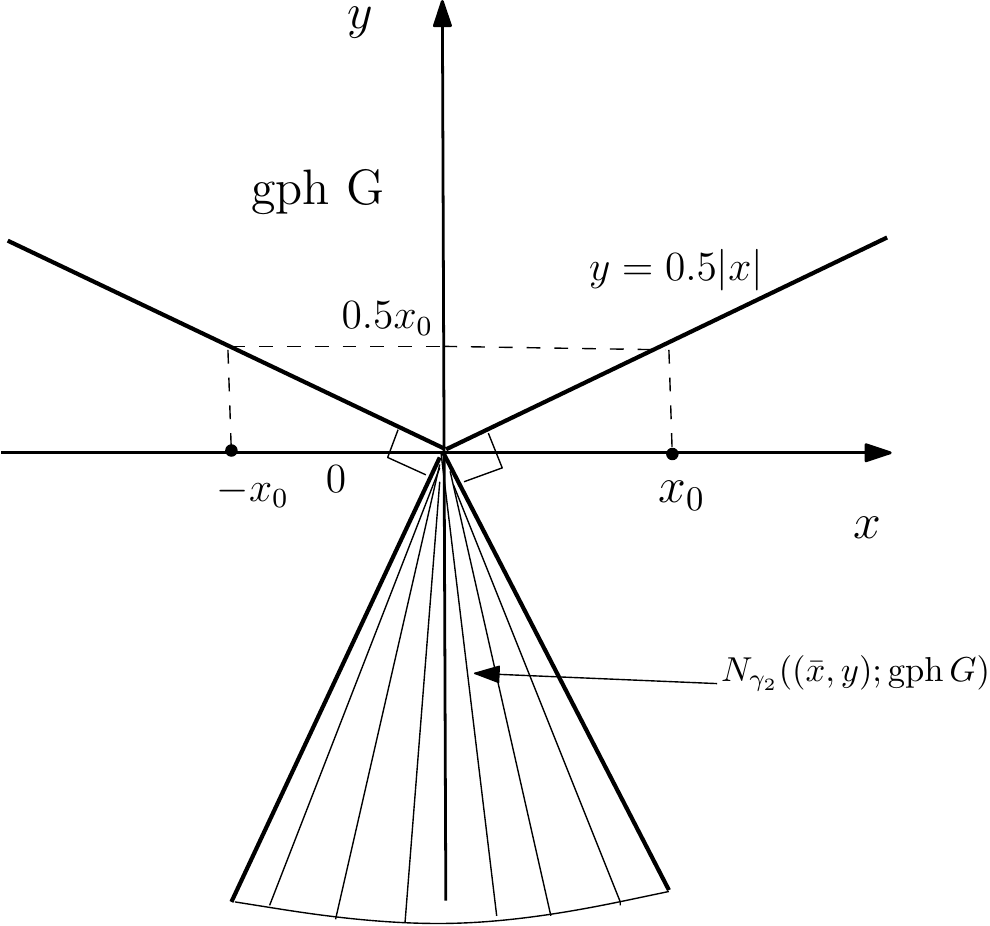}
	\centerline{Figure 2: The $\gamma_2$-normal directions of ${\rm gph\,}G$.}
\end{figure}
So, the right-hand-side of \eqref{Main_formula} can be computed as follows
\begin{align*}
RHS_{\eqref{Main_formula}}= &\bigcap\limits_{\eta \,>\,0} \ \bigcup\limits_{y\, \ge\,  0}\ \bigcup\limits_{(\gamma_1, \gamma_2)\, \in\, \Gamma (\eta+\varepsilon)} \big\{x^*\in X^* \mid (x^*,0)\! \in\! \partial_{\gamma_1}\varphi(\bar x, y) \!+\!N_{\gamma_2}\big ((\bar x,  y); {\rm gph}\, G) \big ) \big\}\\
= &\bigcap\limits_{\eta \,>\,0} \ \bigcup\limits_{(\gamma_1, \gamma_2)\, \in\, \Gamma (\eta+\varepsilon)} \big\{x^*\in X^* \mid (x^*,0)\! \in\! \partial_{\gamma_1}\varphi(\bar x,\bar y) \!+\!N_{\gamma_2}\big ((\bar x, \bar y); {\rm gph}\, G) \big ) \big\}\\
=&  \bigg\{ x^* \in \mathbb{R} \mid (x^*,0)\in \{0\} \times [-1,1] + \left\{(x^*,y^*)\in \mathbb{R}^2 \mid y^* \le - 2|x^*| \right\}
\bigg\}\\
=&  \bigg\{ x^* \in \mathbb{R} \mid \{x^*\}\times [-1,1]\in \big\{(x^*,y^*)\in \mathbb{R}^2 \mid y^* \le - 2|x^*| \big\}
\bigg\}\\
=&\left[-\frac{1}{2},\frac{1}{2}\right].
\end{align*}
This justifies the conclusion of Theorem~\ref{main_result1}.
	}
\end{ex}

{\footnotesize \noindent \textbf{Acknowledgements.} The research of Duong Thi Viet An was supported by Thai Nguyen University of Sciences and the Vietnam Institute for Advanced Study in Mathematics (VIASM). The research of Jen-Chih Yao was supported by  the Grant MOST 105-2221-E-039-009-MY3.
The authors would like to thank Prof. Nguyen Dong Yen for useful comments and suggestions.
}


\begin{thebibliography}{99}
\bibitem{AnYao2016} An, D.T.V., Yao, J.-C.: Further results on differential stability of convex optimization problems. J. Optim. Theory Appl., \textbf{170}, 28--42 (2016)
 \bibitem{AnYen2015} An, D.T.V., Yen, N.D.: Differential stability of convex optimization problems under inclusion constraints. Appl. Anal., \textbf{94}, 108--128 (2015)
 \bibitem{AttouchBrezis1986} Attouch, H., Brezis, H.: Duality for the sum of convex functions in general Banach spaces. In: ``Aspects of Mathematics and its Applications", pp. 125--133, North-Holland Mathematical Library, Vol. \textbf{34} (1986)
 \bibitem{Bonnans_Shapiro_2000}  Bonnans, J.F., Shapiro, A.: Perturbation Analysis of	Optimization Problems. Springer, New York (2000)
 \bibitem{Rockafellar_Brondsted_1965}  Br\o ndsted, A., Rockafellar, R.T.: On the subdifferentiability of convex functions. Proc. Amer. Math. Soc., \textbf{16}, 605--611 (1965)
\bibitem{Hiriart_Urruty_1982} Hiriart-Urruty, J.-B.: $\varepsilon$-subdifferential calculus. Convex analysis and optimization. Res. Notes in Math., \textbf{57}, Pitman, Boston, Mass.-London, 43--92 (1982)
\bibitem{Hiriart_Urruty_1989} Hiriart-Urruty, J.-B.: From convex optimization to nonconvex optimization. Necessary and sufficient conditions for global optimality. Nonsmooth optimization and related topics. Ettore Majorana Internat. Sci. Ser. Phys. Sci., \textbf{43}, Plenum, New York, 219--239 (1989)
\bibitem{Hiriart_Moussaoui_Seeger_1995}  Hiriart-Urruty, J.-B., Moussaoui, M., Seeger, A., Volle, M.: Subdifferential calculus without qualification conditions, using approximate subdifferentials: a survey. Nonlinear Anal., \textbf{24}, 1727--1754 (1995) 
\bibitem{Hiriart_Lemarechal_1993}  Hiriart-Urruty, J.-B., Lemar\'{e}chal, C.: Convex analysis and minimization algorithms. II. Advanced theory and bundle methods. Grundlehren Math. Wiss., Springer-Verlag, Berlin (1993)

\bibitem{Hiriart_Phelps_1993} J.-B. Hiriart-Urruty and R.R. Phelps,  \textit{Subdifferential calculus using $\varepsilon$-subdifferentials}. J. Funct. Anal., \textbf{118}, 154--166 (1993) 
\bibitem{Ioffe_Tihomirov_1979} Ioffe A.D., Tihomirov, V.M.: Theory of Extremal Problems. North-Holland Publishing Company. Amsterdam (1979)

\bibitem{Mordukhovich_2006}  Mordukhovich,  B.S.: Variational Analysis and Generalized Differentiation. Volume~I: Basic Theory, Volume~II: Applications. Springer-Berlin (2006)
\bibitem{Mordukhovich_Nam_Rector_Tran}   Mordukhovich, B.S., Nam, N.M., Rector,  B., Tran, T.: Variational geometric approach to generalized differential and conjugate calculi in convex analysis. Set-Valued Var. Anal., \textbf{25}, 731--755 (2017)
\bibitem{MordukhovichEtAl_2009}  Mordukhovich,  B.S., Nam, N.M., Yen, N.D.: Subgradients of marginal functions in parametric mathematical programming. Math.
Program. Ser. B, \textbf{116}, 369--396 (2009)
\bibitem{MoussaouiSeeger1994} Moussaoui, M., Seeger, A.: Sensitivity analysis of optimal value functions of convex parametric programs with	possibly empty solution sets. SIAM J. Optim.,~\textbf{4}, 659--675 (1994)
\bibitem{Penot2013} Penot, J.-P.: Calculus Without Derivatives. Graduate Texts in Mathematics. Springer, New York (2013)
\bibitem{Rockafellar_1970} Rockafellar,  R.T.: Convex Analysis. Princeton University Press, Princeton (1970)
\bibitem{Seeger1994} Seeger, A.: Approximate Euler-Lagrange inclusion, approximate transversality condition, and sensitivity analysis of convex parametric problems of calculus of variations. Set-valued Anal., 307--325 (1994)
\bibitem{Seeger1996} Seeger, A.: Subgradients of optimal-value functions in dynamic programming: the case of convex systems  without optimal
	paths. Math. Oper. Res., \textbf{21}, 555--575 (1996)

\bibitem{Volle1994} Volle, M.: Calculus rules for global approximate minima
	and applications to approximate subdifferential calculus.  J. Global Optim., \textbf{5}, 131--157 (1994)


\bibitem{Zanlinescu_2002}  Z\u{a}nlinescu, C.: Convex Analysis in General Vector Spaces. World Scientific. New Jersey-London-Singapore-Hong Kong (2002)

\end{thebibliography}
\end{document}